\documentclass{amsart}
\usepackage{graphicx, amssymb, amsmath,amsthm,mathtools,todonotes,verbatim, url} % Required for inserting images

\usepackage[colorlinks=true,linkcolor=cyan,citecolor=cyan]{hyperref}

\newtheorem{theorem}{Theorem}[section]
\newtheorem{lemma}[theorem]{Lemma}

\newtheorem{proposition}[theorem]{Proposition}
\newtheorem{question}[theorem]{Question}
\newtheorem{problem}[theorem]{Problem}

\newtheorem{claim}{Claim}[theorem]
\theoremstyle{definition}
\newtheorem{definition}[theorem]{Definition}

\theoremstyle{remark}

\DeclareMathOperator{\Aut}{Aut}
\DeclareMathOperator{\Subg}{Subg}

\title{The isomorphism problem for finitely generated bi-orderable groups}
\author[F.~Calderoni]{Filippo Calderoni}

\address{Department of Mathematics, Rutgers University, 
Hill Center for the Mathematical Sciences,
110 Frelinghuysen Rd.,
Piscataway, NJ 08854-8019}
\email{filippo.calderoni@rutgers.edu}

\author[A.~Clay]{Adam Clay}

\address{Department of Mathematics, 420 Machray Hall, University of
Manitoba, Winnipeg, MB, R3T 2N2, Canada}
\email{Adam.Clay@umanitoba.ca}
\thanks{The first author is grateful to Marcin Sabok for telling him about the classification problem for finitely generated left-orderable groups and for interesting discussions. We also thank Thomas Koberda for helpful remarks and references to the literature. We thank the anonymous referee and Rapha\"el Carroy for pointing out an inaccuracy in Section~2. Filippo Calderoni was partially supported by the NSF grant DMS -- 2348819.  Adam Clay was partially supported by NSERC grant RGPIN-2020-05343.}
\begin{document}

\begin{abstract}
We analyze the classification problem for finitely generated orderable groups from the viewpoint of descriptive set theory. We analyze the standard Borel space of finitely generated left-orderable groups, and the subspace of finitely generated bi-orderable groups using spaces of relative cones.  We use this setup to show that the isomorphism relation on  finitely generated bi-orderable groups is weakly universal.
\end{abstract}

\maketitle

\section{Introduction}

A \emph{left-ordering} of a group $G$ is a strict total ordering of $G$ that is invariant under left-multiplication, that is,
\(
g<h \implies  fg<fh
\)
for all \(f,g,h\in G\).  A \emph{bi-ordering} of a group is a left-ordering which is also invariant under right-multiplication.  Groups that admit a left-ordering are called \emph{left-orderable}, those which also admit a bi-ordering are \emph{bi-orderable}.  These properties, and properties of the orderings themselves, can also be understood dynamically, as there is a deep connection in infinite group theory between orderability properties of groups and dynamical counterparts.  For example, a countable group is left-orderable precisely when it admits a faithful action by homeomorphisms on the real line. The article~\cite{Nav10} and the monographs~\cite{ClaRol,DerNavRiv} are excellent references for the subject of orderable groups.

In this paper we investigate the classification problem for finitely generated left-orderable and bi-orderable groups.  The collection of left- and bi-orderable groups is very rich, for instance there are continuum many  pair-wise non-isomorphic finitely generated left-orderable groups, and there is a wide variety of techniques available to construct such families of left-orderable groups having prescribed properties \cite{KimKobLod, HydLod, BoyGorHu}. However, this abundance and variety do not obstruct the development of a satisfactory classification of left- or bi-orderable groups, such as the theory developed by Baer for rank one abelian groups, or Ulm's classification of reduced abelian \(p\)-groups employing Ulm invariants.

This paper provides such an obstruction, by proving an anti-classification result for finitely generated bi-orderable groups that excludes the possibility of classifying them up to isomorphism. In fact, we prove
that the isomorphism relation  \(\cong_\mathcal{BO}\) on the space of finitely generated bi-orderable groups is weakly universal, and so is conjecturally as complicated as possible. Consequently, classifying finitely generated left-orderable (or even bi-orderable) groups is conjecturally as unfeasible as classifying all finitely generated groups.
 
Our techniques are founded in descriptive set theory. The basis of our approach, already pointed out in~\cite{GolKunLod}, is  a method of considering the class of all finitely generated left-orderable groups as a standard Borel space. The idea builds on Grigorchuk's construction of the space of marked groups. Then one can use descriptive set theory to compare the relative complexity of the isomorphism equivalence relation on the space of finitely generated left-orderable groups with other equivalence relations. This approach has proven useful in a variety of contexts spanning from ergodic theory~\cite{ForRudWei,ForWei, ForWei22}, to dynamics~\cite{DekKwiPenSab}, functional analysis~\cite{CamGao,FerLouRos, Sab16}, and geometric group theory~\cite{OsiOya, Tho08}.  

More recently, there is an ongoing investigation of applications of descriptive set theory to left-orderable groups~\cite{CalCla22,CalCla24a,CalCla24b, CalSha,GolKunLod,Mol,Pou25, HoKhaRos}, with this manuscript being our latest contribution. Our approach, however, introduces a new application of descriptive set theory, and opens up the possibility of investigating the classification of special sub-classes of orderable groups. 

Our technique of proof is to encode the conjugacy action of a non-abelian free group on its subgroups within the isomorphism relation on the space of bi-orderable groups, see Theorem \ref{thm:main}.  We do this by constructing a particular family of finitely generated bi-orderable groups, all of which arise as the quotient of a single finitely generated bi-orderable group that encodes the conjugacy relation within its automorphism group.  See Section \ref{sec:cocycle}.

\subsection{Organisation of the paper}
In Section~\ref{sec:spaces} we review the basic background required from descriptive set theory and define the standard Borel space of left-orderable (and bi-orderable) finitely generated groups.  We also prove our main result, assuming the existence of a special kind of bi-orderable group. (See Theorem~\ref{thm:main}.) In Section~\ref{sec:cocycle} we discuss the machinery of inhomogeneous 2-cocycles and construct the required group. We conclude our paper with a discussion of open questions in Section~\ref{sec:questions}.

\section{The spaces, and a weak reduction}
\label{sec:spaces}

In this section
we introduce the space of finitely generated left-orderable groups and the space of finitely generated bi-orderable groups, denoted $\mathcal{LO}$ and $\mathcal{BO}$ respectively, as well as the tools from descriptive set theory needed for our arguments. 

Our discussion of $\mathcal{LO}$ and $\mathcal{BO}$ is based on Grigorchuk's space of finitely generated groups~\cite{Gri84,Gri05}. (See also Champetier~\cite{Cha00} or Thomas~\cite{Tho08} for a self-contained presentation of Grigorchuk's space.)

\subsection{Countable Borel equivalence relations}
Recall that a topological space is \emph{Polish}  if it is separable and admits a complete compatible metric. By a classical theorem of Alexandrov, a subspace of the topological Polish space \((X,\tau)\) is Polish if and only if it is \(G_\delta\) in \(\tau\). (E.g., see~\cite[Theorem~3.11]{Kec}.) In particular, a closed subset of a Polish space is Polish with the subspace topology.

Now suppose that \(X\) is a set and that \(\mathcal{B}\) is a \(\sigma\)-algebra of subsets
of \(X\). Then \(( X, \mathcal{B})\) is a \emph{standard Borel space} if there exists a Polish topology \(\tau\) on
\(X\) such that \(\mathcal{B}\) is the \(\sigma\)-algebra generated by \(\tau\). 
In particular, every Polish space is standard Borel with the Borel structure generated by its topology. Moreover, any Borel subset \(A\) of a standard Borel space \((X,\mathcal{B})\) inherits a standard Borel structure \((A,\mathcal{B}_ A)\) where \(\mathcal{B}_A=\{B\subseteq A : B\in \mathcal{B}\}\).

An important example in this paper is the Polish space of all subgroups of a fixed countable group with the Chabauty topology. More precisely, let \(G\) be a countable group and set \(\Subg(G) = \{H \in 2^G : H\text{ is a subgroup of }G\}\).    Recall that the topology on $2^G$  has a subbasis consisting of the sets \[
    U_g=\{S\subseteq G : g\in S\},\qquad U^C_g=\{S\subseteq G : g\notin S\}, 
    \] 
 where $g \in G$.  Then, for example, the subsets of $G$ that are not closed under the group operation comprise an open set in $2^G$, which is given by
 \[ \bigcup_{g, h \in G} U_g \cap U_h \cap U_{gh}^C.
 \]
By similar reasoning, one checks that $2^G \setminus \Subg(G)$ is open, so that \(\Subg(G)\) is closed and therefore a compact Polish space with the subspace topology.  

An equivalence relation \(E\) on a standard Borel space \(X\) is said to be \emph{Borel} if \(E \subseteq X\times X\) is a Borel subset of \(X\times X\). A Borel equivalence relation \(E\) is said to be \emph{countable} if every \(E\)-equivalence class is countable. 

A typical example of such an equivalence relation arises as follows. Let \(G\) be a countable group. Then a \emph{standard Borel \(G\)-space} is a standard Borel space \(X\) equipped with a Borel action \(G\times X\to X \) of \(G\) on \(X\), which we denote by \( (g, x) \mapsto g \cdot x\).
For any \(x\in X\), we let \(G\cdot x =\{g\cdot x : g\in G\}\) denote the orbit of \(x\), and \(E^X_G\) the \emph{orbit equivalence relation} on \(X\) whose classes are the \(G\)-orbits.  That is,
\[
x\mathbin{E}^X_Gy\quad \iff\quad G\cdot x = G\cdot y.
\]
Whenever \(G\) acts on \(X = \Subg(G)\) by conjugation we will instead write \(E_c(G)\) in place of \( E^X_G\).

Let \(E, F\) be countable Borel equivalence relations on standard Borel spaces $X$ and $Y$ respectively. A Borel map \(f \colon X\to Y\) is said to be a \emph{homomorphism} from \(E\) to \(F\)  if 
\[x\mathbin{E} y \implies f(x)\mathbin{F} f(y)\]
for all \(x,y\in X\).  We say that \(E\) is \emph{weakly Borel reducible}
to \(F\), written \(E \leq^w_{B} F\) if and only if there exists a countable-to-one Borel homomorphism
\(f \colon X \to Y\) from \(E\) to \(F\). In this case, we say that \(f\) is a \emph{weak Borel reduction} from
\(E\) to \(F\).  If \(f\) satisfies the stronger property that for all \(x, y \in X\),
\[x\mathbin{E}y \iff f(x)\mathbin{F}f(y),\]
then \(f\) is said to be a \emph{Borel reduction} from \(E\) to \(F\), and we write \(E \leq_B F\). 

\begin{definition}
A countable Borel equivalence relation \(E\) is said to be
\emph{universal} if and only if \(F \leq_B
 E\) for every countable Borel equivalence relation \(F\). Similarly, \(E\) is said to be \emph{weakly
universal} if and only if \(F \leq^w_B E\) for every countable Borel equivalence relation \(F\).
\end{definition}

The following result of Thomas and Velickovic~\cite[Theorem~7]{ThoVel}  will be used later in this paper. (See also Gao~\cite[Theorem~2]{Gao00} for an alternative shorter proof.) Here, and throughout this manuscript, we use $\mathbb{F}_n$ to denote the free group on $n$ generators. We recall that the free group \(\mathbb{F}_n\) is bi-orderable for all \(n\in \mathbb{N}_+\cup \{\infty\}\). (E.g., see~\cite[Theorem~3.4]{ClaRol}.) We will be using this fact a lot without reference.

\begin{theorem}[Velickovic--Thomas~1999]
    For all \(n> 1\), the countable Borel equivalence relation \(E_c(\mathbb{F}_n)\) is universal.
\end{theorem}

Clearly, every universal countable Borel equivalence relation is also weakly universal. However, the following fundamental problem of Hjorth is still open.

\begin{question}[Hjorth]
    Does there exist a weakly universal countable Borel equivalence relation which is not universal?
\end{question}

\subsection{The space of finitely generated groups, and the isomorphism relation}
\label{subsec:fggroups}

Our presentation of the space of left-orderable groups is different from the one of \cite{GolKunLod, Osi21apal}. While it may be possible to follow their setup and recover our results using their space, our approach is particularly convenient for tackling isomorphism problems.

Let \(\mathbb{F}_\infty\) be the free group whose generating set is \(S= \{x_i: i\in \mathbb{N}\}\).
Let \(\mathcal{N} \subset 2^{\mathbb{F}_\infty}\) be the space of all normal subgroups of \(\mathbb{F}_\infty\) that contain all but finitely many elements of $S$.  Any finitely generated group is therefore isomorphic to \(\mathbb{F}_\infty/N\) for some \(N\in \mathcal{N}\).  Observe that $\mathcal{N}$ is Borel as follows:

Every automorphism of \(\mathbb{F}_\infty\) induces a self-homeomorphism of \(\Subg(\mathbb{F}_\infty)\),  from which it follows that the set of all normal subgroups $\mathcal{N}_0$ of $\mathbb{F}_\infty$ (i.e., the collection of all subgroups fixed by inner automorphisms) is closed and therefore compact.  Moreover, for each finite subset $F \subset S$ the set
$$\mathcal{N}_F = \{N \in \mathcal{N}_0 : x_i \in N \iff x_i \notin F \}$$
is clearly $G_\delta$, since
\[ \mathcal{N}_F = \mathcal{N}_0 \cap  \bigcap_{x_i \notin F} U_{x_i} \cap  \bigcap_{x_i \in F} U_{x_i}^C .
\]
We conclude that
\[ \mathcal{N} =  \bigcup_{F \subset S, \; |F|<\infty} \mathcal{N}_{F},
\]
showing that \(\mathcal{N}\) is $F_\sigma$. It follows, in particular, that \(\mathcal{N}\) is a standard Borel space with the subspace Borel structure. So, there is some topology on \(\mathcal{N}\) that is compatible with its standard Borel structure and turns it into a Polish space. However, we point out below that the obvious subspace topology of \(\mathcal{N}\) is not Polish---contrary to what was claimed in an earlier version of this manuscript, and what is sometimes claimed in the literature.

\begin{theorem}
\label{thm : N not Polish}
    The subset $\mathcal{N}$ is not \(G_\delta\) in \(\mathcal{N}_0\). Thus, $\mathcal{N}$ is not Polish  with the subspace topology.
\end{theorem}

The proof uses the following lemma:

\begin{lemma}
\label{lem:not_g_delta}
    Suppose that $X$ is a Baire space and that $U \subset X$ is a dense meagre set.  Then $U$ is not $G_{\delta}$.
\end{lemma}
\begin{proof}
    Let \(X\) be Baire and \(U\) be dense and meagre in \(X\).
    Suppose also that $U = \bigcap_{i=1}^\infty U_i$ where each $U_i$ is open, towards contradiction. Note that each $U_i$ is also dense in $X$.  Now take the complement, and observe
\[ X \setminus U = X \setminus \left( \bigcap_{i=1}^\infty U_i \right)  = \bigcup_{i=1}^\infty X \setminus U_i.
\]
Each set $X \setminus U_i$ is closed and has empty interior, since $U_i$ is dense.  So each set \(X \setminus U_i\) is nowhere dense, and $X \setminus U$ is meagre. But now both $U$ and $X \setminus U$ are meagre, a contradiction.
\end{proof}

\begin{proof}[Proof of Theorem~\ref{thm : N not Polish}]
First, we show that $\mathcal{N}$ is dense in $\mathcal{N}_0$.  To see this, let $N \in \mathcal{N}_0$ and let \(U= U_{g_1}\cap \dotsb \cap U_{g_r}\cap U^C_{g_{r+1}}\cap \dotsb\cap U^C_{g_n}\) be a basic open set with \(N \in U\).

    Let \(k\in \mathbb{N}\) such that \(g_1,\dotsc, g_n\in \langle x_0,\dotsc, x_k \rangle \) and denote by \(q\colon \mathbb{F}_\infty \to \mathbb{F}_\infty/N\) the canonical quotient map onto \(\mathbb{F}_\infty/N\). Consider two cases:
    
    First suppose that there exists \(\ell >k\) such that \(x_\ell \in N\). Then
    fix a non-identity element $h \in \mathbb{F}_{\infty}/N$ and define $\phi \colon \mathbb{F}_{\infty} \rightarrow \mathbb{F}_{\infty}/N$ by 
    \[
\phi(x_i) =
\begin{cases}
    q(x_i) &\text{if $i < \ell$}\\
    h &\text{if $i = \ell$} \\
    id & \text{otherwise.}
\end{cases}
\]
Set $K = \ker \phi$, and note that $K \neq N$, and $K$ contains all but finitely many generators of $\mathbb{F}_\infty$, and \(K \in U\) by construction.  

On the other hand, suppose that for all $\ell >k$, we have $x_\ell \notin N$.  Then define $\phi \colon \mathbb{F}_{\infty} \rightarrow \mathbb{F}_{\infty}/N$ by 
    \[
\phi(x_i) =
\begin{cases}
    q(x_i) &\text{if $i < \ell$}\\
    id &\text{if $i \geq \ell$}. 
\end{cases}
\]
Again, set $K = \ker(\phi)$ and note $K \neq N$, and $K$ contains all but finitely many generators of $\mathbb{F}_\infty$, and \(K \in U\) by construction.  It follows that $\mathcal{N}$ is dense in $\mathcal{N}_0$.

Now it is left to show that \(\mathcal{N}\) is meagre. For $i \geq 1$ set 
$$
X_i = \{ N \in \mathcal{N} \mid x_j \in N \text{ for all } j>i \}.
$$
Observe that 
\[ X_i = \bigcap_{j=i+1}^\infty U_{x_j},
\]
so each $X_i$ is a closed set, and that $\mathcal{N} = \bigcup_{i=1}^\infty X_i$.  

Moveover, each $X_i$ has empty interior:  If $N \in X_i$ and $U$ is a neighbourhood of $N$, by arguments identical to above we can construct a normal subgroup $K \in U$ and such that $K \notin X_i$.  Specifically,  let \(U= U_{g_1}\cap \dotsb \cap U_{g_r}\cap U^C_{g_{r+1}}\cap \dotsb\cap U^C_{g_n}\) be a basic open set with \(N \in U\).

    Let \(k\in \mathbb{N}\) such that \(g_1,\dotsc, g_n\in \langle x_0,\dotsc, x_k \rangle \) and denote by \(q\colon \mathbb{F}_\infty \to \mathbb{F}_\infty/N\) the canonical quotient map onto \(\mathbb{F}_\infty/N\). Fix a non-identity element $h \in \mathbb{F}_{\infty}/N$ and define $\phi \colon \mathbb{F}_{\infty} \rightarrow \mathbb{F}_{\infty}/N$ by 
    \[
\phi(x_i) =
\begin{cases}
    q(x_i) &\text{if $i \leq k$}\\
    h &\text{if $i>k$.} 
\end{cases}
\]
Set $K = \ker(\phi)$, and note that $K \in U$ by construction, but also $K \notin \mathcal{N}$, so certainly $K \notin X_i$.

The result now follows from Lemma~\ref{lem:not_g_delta}.
\end{proof}

Next we consider the equivalence relation $\cong$ on $\mathcal{N}$ given by group isomorphism.

Let \(\Aut_f(\mathbb{F}_\infty)\) be the subgroup of \(\Aut(F_\infty)\) generated by the
elementary Nielsen transformations
\[\{\alpha_i
: i \in\mathbb{N}\} \cup \{\beta_{ij} : i \neq j \in\mathbb{N} \},\]
where \(\alpha_i\)
is the automorphism sending \(x_i\) to \(x
^{-1}
_i\)
and fixing the other generators, and
\(\beta_{ij}\) is the automorphism taking \(x_i\) to \(x_ix_j\) and leaving the other generators fixed.   

\begin{proposition}\cite[Section~3]{Cha00}
    If \(N, M \in\mathcal{N} \), then
\(
\mathbb{F}_\infty/N \cong
F_\infty/M\) if and only if there is \(\phi\in \Aut_f(\mathbb{F}_\infty)\) such that \( \phi(N) = M\).
\end{proposition}

In particular, since \(\Aut_f(\mathbb{F}_\infty)\) is countable, this means that the equivalence classes of $(\mathcal{N}, \cong)$ are countable.  In fact, we have the following.

\begin{theorem}[Thomas-Velickovic~1999]
   The isomorphism relation  \(\cong\) on the space $\mathcal{N}$ of finitely generated groups is a universal countable Borel equivalence relation.
\end{theorem}

\subsection{The space of finitely generated left-orderable groups}
Let \(G\) be a left-orderable group. Recall that every left-ordering $<$ of $G$ is uniquely determined by its \emph{positive cone} $P = \{ g \in G : g>id \}$.  Conversely, if we are given a set $P \subset G$ satisfying $P \cdot P \subset P$ and $G \setminus \{id\} = P \sqcup P^{-1}$, then $P$ is the positive cone of the left-ordering $<$ of $G$ defined by 
\[ g<h \iff g^{-1}h \in P.
\]
Here, and later in the manuscript, we use $\sqcup$ to indicate disjoint union.

A subgroup $C \subset G$ is \emph{left-relatively convex} if there exists a left-ordering $<$ of $G$ such that for all $a, b \in C$ and $g \in G$, if $a<g<b$ then $g \in C$.  This property can also be reworded in terms of the existence of a certain kind of subsemigroup of $G$, similar to a positive cone.

\begin{proposition}
\label{prop : rel cone}
    Suppose that $G$ is a left-orderable group, and let $C \subset G$ be a subgroup.  Then $C$ is left-relatively convex if and only if there exists a semigroup $P \subset G$ such that 
    \begin{enumerate}
        \item $P \sqcup P^{-1} = G \setminus C$, and
        \item $CPC \subset P$.
    \end{enumerate}
    Moreover, if $C$ is normal in $G$ and $q\colon G \rightarrow G/C$ denotes the quotient map, then $q(P)$ is the positive cone of a left-ordering of $G/C$.
\end{proposition}
\begin{proof}
    Suppose that the subgroup $C$ is convex relative to a left-ordering $<$ of $G$ with positive cone $Q$.  Set $P = Q \setminus C$.  

    We check that $P$ is a semigroup.  Given $a, b \in P$ suppose that $ab \notin P$, then $ab \in Q \cap C$.  Now $id < a <ab$ would imply $a \in C$, so we must instead have $id <ab <a$.  But then $b <id$ upon left-multiplying by $a^{-1}$, a contradiction.  We conclude $P$ is a semigroup.

Now to show (1), note 
    \[ P \sqcup P^{-1} = (Q\setminus C) \sqcup (Q\setminus C)^{-1} = (Q \sqcup Q^{-1})\setminus C = G \setminus C,
    \]
where in the last line we use the fact that every positive cone $Q$ of a left-ordering of $G$ satisfies $G\setminus \{id\} = Q \sqcup Q^{-1}$.
   
    To show (2) holds, suppose that $a, b \in C$ and $g \in P$, and that $agb \notin P$.  Note $agb \notin C$ since $g \notin C$, and so our assumption implies that $agb \in P^{-1}$.  But then $a<agb$ is not possible, so instead we must have $agb <a$, implying $gb<id$ and so $b<g^{-1}$.  Now $b<g^{-1}<id$, which implies $g^{-1} \in C$ by convexity, a contradiction.  So $C$ being relatively convex implies the existence of such a $P$.

    Now, suppose we have such a $P$.  Choose a positive cone $Q \subset C$ of a left-ordering of $C$, and set $R = P \cup Q$.  We show that $R$ is a positive cone and $C$ is convex relative to the ordering determined by $R$.

    First, let $a, b \in R$.  If $a, b \in Q$ or $a, b \in P$ then $ab \in R$.  So suppose $a \in Q$ and $b \in P$.  Then $ab \in R$ by (2), and if $a \in P$ and $b \in Q$ then similarly $ab \in R$ by (2).  So $R$ is a semigroup.  Now observe
    \[ R \sqcup R^{-1} = (P \sqcup P^{-1}) \cup (Q \sqcup Q^{-1}) = (G \setminus C) \cup (C \setminus \{id\}) = G \setminus \{id\}.
    \]
    So $R$ is indeed a positive cone.  Denote the associated left-ordering by $<$ and suppose that $a,b \in C$ and $g \in G$, and $a<g<b$.  If $g \notin C$, we conclude $a^{-1}g \in P$ and $g^{-1}b \in P$.  But $P$ is a semigroup, so $(a^{-1}g)(g^{-1}b) = a^{-1}b \in P$, a contradiction since $a^{-1}b \in C$ and $P \cap C = \emptyset$.

Last, suppose $C$ is normal, and consider $q(P) \subset G/C$.  Note that $q(P)$ is a semigroup, and $G/C = q(P) \cup q(P)^{-1}$ by property (1).  We check that $q(P) \cap q(P)^{-1} = \emptyset$ as follows.  If $q(P) \cap q(P)^{-1} \neq \emptyset$ then $q(h') = q(h)$ for some $h' \in P$ and $h \in P^{-1}$.  But then $h^{-1}, h'$ are both elements of $P$, and so $h^{-1}h' \in P$.  Yet $q(h') = q(h)$ implies $h^{-1}h' \in C$, a contradiction.  Thus $G/C = q(P) \sqcup q(P)^{-1}$ and $q(P)$ is the positive cone of a left-ordering.
\end{proof}

Call a subset $P$ as in the previous proposition a \emph{relative cone}.  Let $\mathrm{LO}_{rel}(G) \subset 2^G$ denote the \emph{space of relative cones of \(G\)}, equipped with the subspace topology. 

\begin{proposition}
\label{prop:compactrel}
    For every group $G$, the space $\mathrm{LO}_{rel}(G)$ is compact.
\end{proposition}
\begin{proof}
We show that the complement of $\mathrm{LO}_{rel}(G)$ is open.  To this end, suppose that $S \in 2^G \setminus \mathrm{LO}_{rel}(G)$, meaning either:

\begin{enumerate}
    \item $S$ is not a semigroup, or
    \item $C = G \setminus (S \sqcup S^{-1})$ is not a subgroup, or
    \item $C = G \setminus (S \sqcup S^{-1})$ does not satisfy $CSC \subset S$, or
        \item $S \cap S^{-1} \neq \emptyset$.
\end{enumerate}
Suppose (1) holds, meaning $S$ is not a semigroup, then there exist $a, b \in S$ such that $ab \notin S$, meaning that 
\[ S \in U_a \cap U_b \cap U_{ab}^C.
\]
The set of all subsets of $G$ which are not semigroups is therefore expressible as
\[ \bigcup_{a, b \in G} U_a \cap U_b \cap U_{ab}^C,
\]
which is an open set. 

If (2) holds then $C = G \setminus (S \sqcup S^{-1})$ is not a subgroup, and so either: (i) there exists $a \in C$ such that $a^{-1} \notin C$, (ii) there exist $a, b \in C$ such that $ab \notin C$, (iii) $id \notin C$.

First note that (i) is not possible, because if $a \notin S \sqcup S^{-1}$ then $a^{-1} \notin S \sqcup S^{-1}$.   If $S$ satisfies (ii) then there exist $a, b \notin S \sqcup S^{-1}$ such that $ab \in S \sqcup S^{-1}$, meaning that $S$ lies in the set 
\[ U_a^C \cap U_b^C \cap U_{a^{-1}}^C \cap U_{b^{-1}}^C \cap (U_{ab} \cup U_{(ab)^{-1}}).
\]
Therefore the set of all $S$ satisfying (ii) is the union
\[ \bigcup_{a, b \in G} U_a^C \cap U_b^C \cap U_{a^{-1}}^C \cap U_{b^{-1}}^C \cap U_{ab} \cup U_{(ab)^{-1}},
\]
which is an open set.  Finally, $S$ satisfies (iii) if and only if $S \in U_{id}^C$, which is also an open set.

Next we consider subsets $S$ of $G$ satisfying (3).  In this case, there exist $a, b \notin S \sqcup S^{-1}$ and $c \in S$ such that $acb \notin S$.  This means that $S$ lies in
\[ U_a^C \cap U_b^C \cap U_{a^{-1}}^C \cap U_{b^{-1}}^C \cap U_c \cap U_{acb}^C, 
\]
and the set of all $S$ satisfying (3) is precisely
\[ \bigcup_{a, b, c \in G} U_a^C \cap U_b^C \cap U_{a^{-1}}^C \cap U_{b^{-1}}^C \cap U_c \cap U_{acb}^C,
\]
which is an open set.

Last, subsets of \(G\) which satisfy (4) are those that lie in the open set
\[ \bigcup_{a \in G} U_a \cap U_{a^{-1}}.
\]

This shows that the complement of $\mathrm{LO}_{rel}(G)$ is open, so the space is compact.
\end{proof}

As a remark, note that $\emptyset \in \mathrm{LO}_{rel}(G)$.  If we add to Proposition~\ref{prop : rel cone} the requirement that $P$ be nonempty, then Antolin and Rivas~\cite{AntRiv} show that $\mathrm{LO}_{rel}(G)$ is no longer compact.  

Given a group $G$, we denote by $\mathrm{Conv}(G) \subset \mathrm{Subg}(G)$ the subspace of all left-relatively convex subgroups of $G$. 

\begin{theorem}
\label{thm:compactconvex}
Let $G$ be a left-orderable group.  The surjective map 
\[ \Phi \colon \mathrm{LO}_{rel}(G) \rightarrow \mathrm{Conv}(G)
\]
given by $\Phi(P) = G\setminus(P \sqcup P^{-1})$ is continuous, in particular, $\mathrm{Conv}(G)$ is compact.
\end{theorem}
\begin{proof}
    We show that the preimage of any subbasic open set is open. %, with the subbasic open sets in $\mathrm{Conv}(G)$ being $U_g \cap \mathrm{Conv}(G)$, $U_g^C \cap \mathrm{Conv}(G)$.
Suppose that $\Phi(P) \in U_a$.  This happens if and only if $a \in G \setminus (P \sqcup P^{-1})$, equivalently, $a \notin P$ and $a \notin P^{-1}$.  We conclude that 
    \[\Phi^{-1}(U_a) = U_a^C \cap U_{a^{-1}}^C \cap \mathrm{LO}_{rel}(G),
\]
which is open in $\mathrm{LO}_{rel}(G)$.
Similarly suppose  $\Phi(P) \in U_a^C$, meaning $a \notin G \setminus (P \sqcup P^{-1})$.  This happens if and only if $a \in P$ or $a \in P^{-1}$, meaning 
   \[\Phi^{-1}(U_a^C) = (U_a \cup U_{a^{-1}} )\cap \mathrm{LO}_{rel}(G),
\]
    which is again open so that $\Phi$ is continuous.  Note that $\Phi$ is surjective by Proposition~\ref{prop : rel cone}, so compactness of $\mathrm{Conv}(G)$ follows.
\end{proof}

It is a standard result that $C \subset G $ is left-relatively convex if and only if the left cosets of $C$ admit a total ordering that is invariant under left-multiplication. (See~\cite[Chapter~2]{ClaRol}.)  Moreover, if $\phi : G \rightarrow G$ is an automorphism and $P \subset G$ is a relative cone, then $\phi(P)$ is also a relative cone.  It follows that $C$ is left-relatively convex if and only if $\phi(C)$ is left-relatively convex. We therefore define the \emph{space of finitely generated left-orderable groups} to be 
\[ \mathcal{LO} = \mathcal{N} \cap \mathrm{Conv}(\mathbb{F}_\infty).
\]
By Theorem \ref{thm:compactconvex} and the discussion of Subsection \ref{subsec:fggroups}, $\mathcal{LO}$ is standard Borel, and isomorphism of groups defines a countable Borel equivalence relation on $\mathcal{LO}$. 

\subsection{The space of finitely generated bi-orderable groups}

We can bootstrap the results of the previous section to deal with finitely generated bi-orderable groups, as follows.  First, we note that every bi-ordering $<$ determines a positive cone $P = \{ g \in G : g>id\}$ which is conjugation invariant, that is, $gPg^{-1} \subset P$ for all $g \in G$.  Conversely, a positive cone $P \subset G$ which is conjugation invariant determines a bi-ordering of $G$ via $g<h \iff g^{-1}h \in P$.

\begin{proposition}
\label{prop:birel}
Let $G$ be a group, and $N$ a normal subgroup of $G$.  Then $G/N$ is bi-orderable if and only if there exists $P \in \mathrm{LO}_{rel}(G)$ such that $N = G \setminus (P \sqcup P^{-1})$ and $gPg^{-1} = P$ for all $g \in G$.
\end{proposition}
\begin{proof}
    Let $q \colon G \rightarrow G/N$ denote the quotient map, and suppose that $G/N$ is bi-orderable with positive cone $Q \subset G/N$.  Set $P = q^{-1}(Q)$, and note that $P$ is a semigroup with $NPN \subset P$.  Moreover, as  $G/N = Q \sqcup Q^{-1} \sqcup \{ id \}$, we conclude that $N = G \setminus (P \sqcup P^{-1})$, so $P \in \mathrm{LO}_{rel}(G)$.  Finally, if $h \in P$ and $g \in G$ then $q(h) \in Q$ and $q(g) \in G/N$ satisfy $q(g)q(h)q(g)^{-1} \in Q$.  But then $ghg^{-1} \in P$, as desired.

    Conversely, suppose $P \in \mathrm{LO}_{rel}(G)$ such that $N = G \setminus (P \sqcup P^{-1})$ and $gPg^{-1} = P$ for all $g \in G$.  Then by Proposition \ref{prop : rel cone}, the set $q(P)$ is the positive cone of a left-ordering of $G/N$.  But $gPg^{-1} = P$ for all $g \in G$ implies that $hq(P)h^{-1} = q(P)$ for all $h \in G/N$, so that $q(P)$ is the positive cone of a bi-ordering.
\end{proof}

Given a group $G$, let $\mathrm{BN}(G) \subset \mathrm{Subg}(G)$ denote the subspace of all normal subgroups $N$ of $G$ such that $G/N$ is bi-orderable.

\begin{theorem}
Given a group $G$, the space $\mathrm{BN}(G)$ is compact.
\end{theorem}
\begin{proof}
Given an automorphism $\phi: G \rightarrow G$, one can check that the induced map $\mathrm{LO}_{rel}(G) \rightarrow \mathrm{LO}_{rel}(G)$ given by $P \mapsto \phi(P)$ is a homeomorphism.  The set 
$$\mathrm{BO}_{rel}(G)  = \{ P \in \mathrm{LO}_{rel}(G) : gPg^{-1} = P \mbox{ for all } g \in G\}$$
is precisely the fixed point set of all such homeomorphisms that are induced by inner automorphisms of $G$.  It is therefore a closed subset of $\mathrm{LO}_{rel}(G)$, and so is a compact space.  Recalling the continuous map $ \Phi \colon \mathrm{LO}_{rel}(G) \rightarrow \mathrm{Conv}(G)$ from Theorem \ref{thm:compactconvex}, by Proposition \ref{prop:birel} the restriction 
\[ \Phi\colon \mathrm{BO}_{rel}(G) \rightarrow \mathrm{BN}(G)
\]
is surjective, showing that $\mathrm{BN}(G)$ is compact.
\end{proof}

As in the previous section, we can therefore define the \emph{space of finitely generated bi-orderable groups} to be 
\[ \mathcal{BO} = \mathcal{N} \cap \mathrm{BN}(\mathbb{F}_\infty).
\]
Similar to the case of $\mathcal{LO}$, this space is standard Borel, and isomorphism of groups defines a countable Borel equivalence relation on $\mathcal{BO}$.

Note, however, that neither $\mathcal{BO}$ nor $\mathcal{LO}$ is closed in $2^{\mathbb{F}_\infty}$.  To see this, let $\phi_k\colon \mathbb{F}_\infty \rightarrow \mathbb{Z}^k$ be the homomorphism given by 

\[
\phi_k(x_i) =
\begin{cases}
    (0, \ldots, 1, \ldots, 0) &\text{if \( 1 \leq i \leq k\)}\\
    0 &\text{otherwise};
\end{cases}
\]
where the one in the first expression appears in the $i$-th position.  Set $N_k = \ker \phi_k$, and note that $N_k \in \mathcal{BO}$ for all $k \geq 1$.  However, the sequence $\{ N_k\}$ converges (in $2^{\mathbb{F}_\infty}$) to the commutator subgroup $[\mathbb{F}_\infty, \mathbb{F}_\infty]$ whose quotient is %$\mathbb{Z}^\infty$
the infinite rank free abelian group, and so does not lie in $\mathcal{BO}$, nor in $\mathcal{LO}$.

\subsection{Weak universality of isomorphism of bi-orderable groups}

In this section we prove our main theorem, assuming the existence of a group having certain special properties.  We construct a group having these properties in the next section.

\begin{theorem}
\label{thm:main}
    Let $F$ be a finitely generated nonabelian free group, and $P \subset F$ the positive cone of a bi-ordering.  Suppose that $H$ is a finitely generated group, and that the centre $Z(H)$ contains a free abelian group with free generators $\{a_g\}_{g \in P}$.  Given a nonempty subset $S \subset P$, let $A_S$ denote the subgroup of $Z(H)$ generated by $\{a_g\}_{g \in S}$.  Suppose further that:
    \begin{enumerate}
        \item
        \label{assumption1}
        The group $H/A_S$ is bi-orderable for all $S \subset P$, and
        \item 
        \label{assumption2}
        If $S', S \subset P$ and there exists $h \in F$ such that $S' = hSh^{-1}$ then $H/A_S \cong H/A_{S'}$.
    \end{enumerate}
    Then $(\mathcal{BO}, \cong)$ is weakly universal.
\end{theorem}
\begin{proof}
Suppose that $H$ has generators $\{h_1, \ldots, h_n\}$, and let \(\theta\colon \mathbb{F}_\infty \to H\) denote the homomorphism defined by $\theta(x_i) = h_i$ for $1  \leq i \leq n$, and $\theta(x_i) = id$ for $i > n$. For any subgroup \(G\in \Subg(F)\) consider the quotient map \(q_{P\cap G}\colon H\to H/A_{P\cap G}\). Then define
\begin{align*}
    f\colon \Subg(F) &\to \mathcal{BO}\\
    G & \mapsto N_G=\ker(q_{P\cap G}\circ \theta).
\end{align*}

Note that by our choice of homomorphism $\theta$, we have \(N_G \in \mathcal{N}\). In fact, assumption~\eqref{assumption1} ensures that \(H/A_{P\cap G}\) is a bi-orderable group, therefore \(N_G\in \mathcal{BO}\).

\begin{claim}
\label{claim : wB}
    The map \(f\) sending a subgroup $ G \in \Subg(F)$ to \(N_G\) is weak Borel reduction from \(\
E_c(F)\) to \(\cong_\mathcal{BO}\).
\end{claim}

\begin{proof}[Proof of Claim~\ref{claim : wB}]
    It is clear from the definition that \(f\) is one-to-one.
    To see it is Borel, we use the subbases 
    \[
    U_g=\{N \in \mathcal{BO}:g\in N \}\qquad U^C_g=\{N \in \mathcal{BO}:g\notin N\}
    \]
    and
     \[
    V_g=\{G \in \Subg(F):g\in G \}\qquad V^C_g=\{G \in \Subg(F):g\notin G\}
    \]
    of $\mathcal{BO}$ and $\Subg(F)$ respectively to show that $f$ is in fact continuous.  Now, fixing $g \in \mathbb{F}_\infty$, note that 
     \begin{align*}
        f^{-1}(U_g) =& \{G \in  \Subg(F) : g\in N_G\}\\
        =&\{G \in  \Subg(F) : g\in \ker (q_{P\cap G}\circ \theta)\}\\
        =&\{G \in  \Subg(F) : \theta(g)\in \ker(q_{P\cap G})\},
    \end{align*}
and similarly $f^{-1}(U_g^C) = \{ G \in \Subg(F):\theta(g) \notin \ker(q_{P \cap G})\}$.  Using this, we consider cases.

    \noindent \textbf{Case 1.} $\theta(g) \notin A_P$. Then no $G \in \Subg(F)$ satisfies $\theta(g)\in \ker(q_{P\cap G}) = A_{P \cap G} \subset A_P$, so $f^{-1}(U_g) = \emptyset$ and $f^{-1}(U_g^C) = \Subg(F)$.

\noindent \textbf{Case 2.} $\theta(g) = id$.  Then every $G \in \Subg(F)$ satisfies $\theta(g)\in \ker(q_{P\cap G})$, so $f^{-1}(U_g) = \Subg(F)$ and $f^{-1}(U_g^C) = \emptyset$.

\noindent \textbf{Case 3.}  $\theta(g) \in A_P \setminus \{id\}$, suppose $\theta(g) = \sum_{i=1}^k n_i a_{g_i}$ where $g_1, \ldots, g_k \in P$ and $n_i \neq 0$ for all $i$.  Then $\theta(g) \in  A_{P \cap G}$ if and only if $a_{g_1}, \ldots, a_{g_k} \in A_{P \cap G}$, which happens if and only if $g_1, \ldots, g_k \in G$.  This is equivalent to $G \in V_{g_1} \cap \ldots \cap V_{g_k}$.  Thus $f^{-1}(U_g) = V_{g_1} \cap \ldots \cap V_{g_k}$ and $f^{-1}(U_g^C) = V_{g_1}^C \cup \ldots \cup V_{g_k}^C$.

    To prove the claim, it therefore remains to prove that if \(G_1\) and \(G_2\) are conjugate subgroups of $F$, then $\mathbb{F}_\infty/N_{G_1} \cong \mathbb{F}_\infty/N_{G_2}$, or equivalently, \(H/A_{P\cap G_1} \cong H/A_{P\cap G_2}\). To see this let \(G_2 = hG_1 h^{-1}\) for some \(h\in F\). Then,
    \(P\cap G_2 = P\cap hG_1h^{-1}\) becomes
    \(P\cap G_2 = h(P\cap G_1)h^{-1}\) because \(P\) is the positive cone of a bi-order, and so it is invariant under conjugation. Then the desired property follows from assumption~\eqref{assumption2}.  This concludes the proof of the claim.
\end{proof}

Since \(E_c(F)\) is universal, we conclude that \(\cong_\mathcal{BO}\) is weakly universal as desired.
\end{proof}

\section{Constructing the required group}
\label{sec:cocycle}

Let $A$ be an abelian group, and $G$ a group.  We begin with a brief review of a well-known construction of a central extension of $G$ by $A$, which we will use to construct a group $H$ having the properties required by Theorem \ref{thm:main}.

A normalized, inhomogeneous 2-cocycle is a function $f : G^2 \rightarrow A$ satisfying \begin{enumerate}
    \item $f(id, g) = f(g, id) = 0$ for all $g \in G$, and
    \item \label{2-cocycle:2} $f(h, k) - f(gh, k) + f(g, hk) - f(g,h) = 0$ for all $g, h, k \in G$.
\end{enumerate}
  From such an $f$ we define a central extension $G_f$ of $G$ by $A$, by setting $G_f = A \times G$ as a set, and equipping it with the operation
\[ (a, g) (b, h) = (a+b+f(g,h), gh).
\]
One checks that $G_f$ is a group, and that $A \times \{id_G\}$ is central.  Computations later in this section will rely upon the fact that $f(g^{-1}, g) = f(g, g^{-1})$ for all $g \in G$, which we can see by applying~\eqref{2-cocycle:2} above to the triple of elements $g, g^{-1}, g \in G$. From this, one also computes that $(a, g)^{-1} = (-a-f(g,g^{-1}), g^{-1})$.

Elements of $H^2(G; A)$ are represented by normalized, inhomogeneous $2$-cocycles $f \colon G^2 \rightarrow A$. While not needed here, the construction above establishes a bijection between elements of $H^2(G; A)$ and equivalence classes of central extensions
\[ 1 \longrightarrow A \longrightarrow H \longrightarrow G \longrightarrow 1,
\]
See \cite[Chapter 4]{brown94} for more details.

One can create automorphisms of $G_f$ as follows.  

\begin{lemma}
\label{lem:aut1}
Suppose that $\phi_1 \colon A \rightarrow A$ and $\phi_2 \colon G \rightarrow G$ are automorphisms, and $f\colon G^2 \rightarrow A$ is a normalized, inhomogeneous 2-cocycle.  Then $\phi(a, g) = (\phi_1(a), \phi_2(g))$ defines an automorphism  $\phi \colon G_f \rightarrow G_f$ if and only if $\phi_1, \phi_2$ satisfy 
\[ \phi_1(f(g, h)) = f(\phi_2(g), \phi_2(h)).
\]
\end{lemma}
\begin{proof}Observe that $\phi$ is a homomorphism if and only if
\[ \phi((a, g)(b, h)) = \phi(a+b+f(g,h), gh) = (\phi_1(a+b+f(g,h)), \phi_2(gh))
\]
and
\begin{multline*}
    \phi(a,g)\phi(b,h) = (\phi_1(a), \phi_2(g))(\phi_1(b), \phi_2(h)) =\\ (\phi_1(a)+\phi_1(b)+f(\phi_2(g), \phi_2(h)),\phi_2(g) \phi_2(h) )
\end{multline*}
are equal, which holds if and only if $\phi_1(f(g, h)) = f(\phi_2(g), \phi_2(h))$.  Note also that $\phi$ is injective (resp. surjective) if and only if both $\phi_1$ and $\phi_2$ are injective (resp. surjective).
\end{proof}

\subsection{Central extensions from left-orderable groups}

This construction is inspired by \cite[Exercise 2.E.35]{loh17}.  Let \(G\) be a countable left-orderable group and fix a positive cone \(P\), with the associated left-ordering $<_P$ of $G$.

Let \(A\) be the free abelian group generated by \(P\), with free abelian generators $\{a_g\}_{g \in P}$, and let \(B\) be the free abelian group generated by \(G\), with free abelian generators $\{b_g\}_{g\in G}$. 

For any two generators \(b_g, b_h \)  of \(B\) set
\[
f(b_g, b_h) =
\begin{cases}
    a_{g^{-1}h} &\text{if \(g<_Ph\)}\\
    0_A &\text{otherwise}.
\end{cases}
\]
The elements of $A$ and \(B\) can be expressed as linear combinations $\sum_j s_ja_{h_j}$ and \(\sum_i t_ib_{g_i}\) respectively for \(t_i, s_j\in \mathbb{Z}\), and $h_j \in P$,  \(g_i \in G\). It is therefore convenient to think of $A$ and \(B\) as a \(\mathbb{Z}\)-modules, and extend \(f\) to a function \(f  \colon B\times B \to A\) by linearity, so that \(f\) is bilinear. It follows that \(f\) is a normalized, inhomogeneous $2$-cocycle.

Now, let $B_f$ denote the central extension of $B$ by $A$ whose underlying set is $A \times B$, equipped with multiplication
\[ (a, b)(a',b') = (a+a'+f(b,b'), b+b').
\]

Note that the group \(G\) acts on both $B$ and $B_f$
by left-multiplication on the indices of elements in $B$.
For \(g\in G\) and $\sum_j t_jb_{g_j} \in B$ we define
\[ g \cdot \sum_j t_jb_{g_j} = \sum_j t_jb_{gg_j},
\]
and for \(\left(\sum_i s_ia_{g_i}, \sum_j t_jb_{g_j}\right) \in B_f\) we define
\[ g 
\cdot\left(\sum_i s_ia_{g_i}, \sum_j t_jb_{g_j}\right) = \left(\sum_i s_ia_{g_i}, g \cdot \sum_j t_jb_{g_j}\right) = \left(\sum_i s_ia_{g_i}, \sum_j t_jb_{g g_j}\right).
\]

This prescription clearly defines a left action of $G$ on $B_f$.  Let $H(G, P)$ denote the semidirect product $B_f \rtimes G$ constructed using this action; we can therefore write the elements of $H(G,P)$ as
\[
\left(\left(\sum_i s_i a_{h_i} , \sum_j t_j b_{g_j} \right), g \right)
\]
where $s_i, t_j \in \mathbb{Z}$, $h_i \in P$, and $g, g_i \in G$.

Note that\footnote{Here we identify \(A\) with \(\{((a,0_B), id_G)\in H(G,P):a\in A\}\).} \(A\) is central in \(H(G,P)\),
because the action of \(G\) on \(B_f\) that we use in creating the semidirect product is
trivial upon restriction to \(A \times \{ 0_B \} \subset B_f \), which itself is central in \(B_f\). 

\begin{proposition}
   If $G$ is a finitely generated left-orderable group, then the group $H(G,P)$ is finitely generated.
\end{proposition}
\begin{proof}
    To see this, we first note that the following is a generating set for $H(G,P)$:
\[ \{ ((a_h, 0_B), id_G)\}_{h \in P} \cup \{ ((0_A, b_g), id_G)\}_{g \in G} \cup \{((0_A, 0_B), e_i) \}_{i=1}^n
\]
where $e_i$, for \(i=1,\dotsc,n\), denote generators for \(G\).  Moreover, in the group $B_f$ the following identity holds:
\begin{multline*}
     [(0_A, b_g), (0_A, b_h)] = (0_A, b_g)(0_A, b_h)(f(b_g, b_g), -b_g)(f(b_h, b_h), -b_h) =\\
     (f(b_g, b_h)-f(b_h, b_g), 0_B).
\end{multline*}

Thus, for an arbitrary $h \in P$ we have $[(0_A, b_0), (0_A, b_h)] = (a_h, 0_B)$, meaning that $\{ (0_A, b_g) \}_{g \in G}$ constitutes a generating set of $B_f$, so that 
\[\{ ((0_A, b_g), id_G)\}_{g \in G} \cup \{((0_A, 0_B), e_i) \}_{i=1}^n
\]
is a generating set of $H$.  We further calculate that for all $g \in G$:
\[ ((0_A,0_B), e_i)((0_A, b_g), id_G)((0_A,0_B), e_i)^{-1} = 
((0_A, b_{e_ig}), id_G).
\]
So in fact, since $\{e_i\}_{i=1}^n$ generate $G$, it suffices to take 
\[\{ ((0_A, b_{id_G}), id_G)\} \cup \{((0_A, 0_B), e_i) \}_{i=1}^n
\]
as a generating set of $H$, which is finite.
\end{proof}
%\subsection{Automorphisms of \(H\)}

In order to create automorphisms of $H(G,P)$, we begin with a brief observation about semidirect products in general.  The proof is a straightforward computation, so we omit it.

\begin{lemma}
\label{lem:aut2}
Suppose that $N$ and $K$ are groups, and that $K$ acts on $N$ from the left, with the action being denoted by $k \cdot n$ for all $k \in K$ and $n \in N$.  If $\psi_1 \colon N \rightarrow N$ and $\psi_2 \colon K \rightarrow K$ are automorphisms of $N$ and $K$ respectively, then $\psi(n, k) = (\psi_1(n), \psi_2(k))$ defines an automorphism $\psi\colon N \rtimes K \rightarrow N \rtimes K$ if and only if the automorphisms $\psi_1$ and $\psi_2$ satisfy
\[ \psi_1(k \cdot n) = \psi_2(k) \cdot \psi_1(n)
\]
for all $n \in N$ and $k \in K$.
\end{lemma}

In the next lemma, we use the notation $\Aut(G, <_P)$ to denote the group of all automorphisms $\phi\colon G \rightarrow G$ such that $g<_Ph$ if and only if $\phi(g)<_P \phi(h)$ for all $g, h \in G$---that is, the group of all automorphisms that preserve the ordering $<_P$.

%Let \(\phi\) be an automorphism of an ordered group \((G,<_P)\).

\begin{lemma}
\label{lem:aut3}
Let $G$ be a finitely generated left-orderable group.
    There is an embedding \(\Psi\colon \Aut(G, <_P) \to \Aut(H(G, P))\), given by $\Psi(\phi) =\widetilde \phi$ where
    $$\widetilde\phi((a_h, 0_B), id_G) = ((a_{\phi(h)}, 0_B), id_G)$$ for all \(h\in P\).
\end{lemma} 
\begin{proof}    
Given $\phi \in \Aut(G, <_P)$, define $\phi_1\colon A \rightarrow A$ and $\phi_2 \colon B \rightarrow B$ by setting $\phi_1(a_h) = a_{\phi(h)}$ and $\phi_2(b_g) = b_{\phi(g)}$, and then extending linearly (again, regarding $A$ and $B$ as $\mathbb{Z}$-modules). Note that $h \in P$ if and only if $\phi(h) \in P$, and so $\phi_1$ is indeed an automorphism of $A$ since it sends a free basis of $A$ to itself.  By similar reasoning, $\phi_2$ is an automorphism of $B$.

Observe that for all $b_g, b_h \in B$, we have 
\[
\phi_1(f(b_g, b_h)) =
\begin{cases}
    a_{\phi(g^{-1}h)} &\text{if \(g<_Ph\)}\\
    0_A &\text{otherwise}, 
\end{cases}
\]
and
\[
f(\phi_2(b_g), \phi_2(b_h)) =f(b_{\phi(g)}, b_{\phi(h)}) = 
\begin{cases}
    a_{\phi(g^{-1}h)} &\text{if \(\phi(g)<_P \phi(h)\)}\\
    0_A &\text{otherwise}. 
\end{cases}
\]
These quantities are always equal, since $\phi(g) <_P \phi(h)$ if and only if $g <_P h$.  By linearity, we conclude that $\phi_1(f(b, b')) = f(\phi_2(b), \phi_2(b'))$ for all $b, b' \in B$.  So, by Lemma \ref{lem:aut1} the expression $\phi'(a,b) = (\phi_1(a), \phi_2(b))$ defines an automorphism $\phi' \colon B_f \rightarrow B_f$.

Next, considering the action of $G$ on $B_f$, we observe that
\begin{multline*} \phi'\left(g \cdot \left(\sum_i s_i a_{h_i} , \sum_j t_j b_{g_j} \right)\right) = \left(\sum_i s_i a_{\phi(h_i)} , \sum_j t_j b_{\phi(gg_j)} \right) \\
= \phi(g) \cdot \phi'\left(\sum_i s_i a_{h_i} , \sum_j t_j b_{g_j} \right).
\end{multline*}
Thus, by Lemma \ref{lem:aut2}, $\widetilde{\phi} : H(G, P) \rightarrow H(G,P)$ defined by $\widetilde{\phi}(b,g) = (\phi'(b), \phi(g))$ for all $b \in B_f$ and $g \in G$ is an automorphism of $H$. Writing the expression for $\widetilde{\phi}$ in full, we have
\[
 \widetilde\phi \left(\left(\sum_i s_i a_{h_i} , \sum_j t_j b_{g_j} \right), g \right) = \left(\left(\sum_i s_i a_{\phi(h_i)} , \sum_j t_j b_{\phi(g_j)} \right), \phi(g) \right), 
 \]
from which it is easy to see that $\widetilde\phi((a_h, 0_B), id_G) = ((a_{\phi(h)}, 0_B), id_G)$ as claimed, and that \(\Psi\colon \Aut(G, <_P) \to \Aut(H(G, P))\) is an injective homomorphism.
\end{proof}

Next, we want to check that certain quotients of $H(G,P)$ are bi-orderable.  Our argument will use a lexicographic construction, and so we need the following lemma. The proof is standard and so we omit it.

\begin{lemma}
\label{lem:boses}
Suppose that 
\[ 1 \longrightarrow K \stackrel{i}{\longrightarrow} H \stackrel{q}{\longrightarrow} G \longrightarrow 1
\]
is a short exact sequence.  If $P_G \subset G$ and $P_K \subset i(K)$ are positive cones of bi-orderings of $G$ and $i(K)$ respectively, and if $hP_Kh^{-1} \subset P_K$ for all $h \in H$, then 
\[ P_H = P_K \cup q^{-1}(P_G)
\]
is the positive cone of a bi-ordering of $H$.
\end{lemma}

\begin{lemma}
\label{lem:BOsemidirectproduct}
Recall that $G$ acts on $B$ from the left according to the rule \[ g \cdot \sum_j t_jb_{g_j} = \sum_j t_jb_{gg_j}
\]    
for all $g \in G$ and $\sum_j t_jb_{g_j} \in B$.  If $G$ is a bi-orderable group, then the semidirect product $B \rtimes G$ is bi-orderable.
\end{lemma}
\begin{proof}
Fix a positive cone $P \subset G$ of a bi-ordering $<_P$ of $G$. Define a positive cone $Q \subset B$ as follows: Given $b = \sum_{j=1}^n t_j b_{g_j}$ with $t_1, \ldots, t_n$ all nonzero, suppose $g_{i_0} = \min_{<_P}\{g_1, \ldots, g_n\}$, and declare $b \in Q$ if and only if $t_{i_0}>0$.  Since $B$ is abelian, $Q$ is the positive cone of a bi-ordering.

Next, we note that $Q$ is preserved by the action of $G$ on $B$.  For if $b = \sum_{j=1}^n t_j b_{g_j}$ as above lies in $Q$, then $t_{i_0} >0$.  Considering $g \cdot b =  \sum_{j=1}^n t_j b_{gg_j}$, note that $gg_{i_0} = \min_{<_P}\{gg_1, \ldots, gg_n\}$ because the ordering $<_P$ is preserved by left-multiplication.  Therefore $t_{i_0}>0$ implies $g \cdot b \in Q$ as well.

Considering the short exact sequence
\[ 1 \longrightarrow B \longrightarrow B \rtimes G \longrightarrow G \longrightarrow 1,
\]
since $B$ admits a bi-ordering whose positive cone $Q$ is invariant under conjugation, Lemma \ref{lem:boses} implies that $B \rtimes G$ is bi-orderable.
\end{proof}

\begin{proposition}
\label{lem:boquotient}
Suppose that $G$ is bi-orderable, that $P \subset G$ is the positive cone of a bi-ordering, and fix a subset $S \subset P$.  Let $A_S$ denote the subgroup of $H(G,P)$ generated by $\{ ((a_h, 0_B), id_G):h \in S\}$.   Then $H(G,P)/A_S$ is bi-orderable.
\end{proposition}
\begin{proof}
Consider the short exact sequence
\[ 1 \longrightarrow A/A_S \stackrel{\iota}{\longrightarrow} H(G,P)/A_S \stackrel{q}{\longrightarrow} H(G,P)/A \longrightarrow 1.
\]
The kernel of this short exact sequence is bi-orderable, since $A/A_S$ is a torsion-free abelian group.  The quotient $H(G,P)/A \cong B \rtimes G$ is also bi-orderable, by Lemma~\ref{lem:BOsemidirectproduct}.

Note that any choice of positive cone $Q \subset \iota(A/A_S)$ will satisfy $hQh^{-1} =Q$, since $\iota(A/A_S)$ is central in $H(G,P)/A_S$.  Thus $H(G,P)/A_S$ is bi-orderable by Lemma~\ref{lem:boses}.
\end{proof}

\begin{proposition}
There exists a group $H$ having the properties required by Theorem \ref{thm:main}.
\end{proposition}
\begin{proof} 
Let $F$ denote a finitely generated free group, and $P \subset F$ the positive cone of a bi-ordering.  

Consider the group $H(F,P)$.  The centre of $H(F, P)$ contains a free abelian group with free generating set $((a_g, 0_B), id_F)$ for $g \in P$.  Given $S \subset P$, and denoting by $A_S$ the subgroup of $A$ generated by $\{((a_g, 0_B), id_F):g \in S\}$, we have that $H(F,P)/A_S$ is bi-orderable by Lemma \ref{lem:boquotient}.

Last, suppose that $S, S' \subset P$ and there exists $h \in F$ such that $S' = hSh^{-1}$.  For each $h \in F$, let $\phi_h \colon F \rightarrow F$ denote the conjugation map $\phi_h(g) = hgh^{-1}$. Then $\phi_h \colon F \rightarrow F$ preserves the ordering $<_P$, and so yields an automorphism $\widetilde{\phi_h} :H(F,P) \rightarrow H(F,P)$ by Lemma \ref{lem:aut3}.  Moreover, 
\[\widetilde{\phi_h}((a_g, 0_B), id_F) = ((a_{\phi_h(g)}, 0_B), id_F) = ((a_{hgh^{-1}}, 0_B), id_F)
\]
and thus $\widetilde{\phi_h}$ maps the generating set $\{((a_g, 0_B), id_F):g \in S\}$ of $A_S$ to the generating set $$\{((a_{hgh^{-1}}, 0_B), id_F):g \in S\} = \{((a_g, 0_B), id_F):g \in S'\}$$ of $A_{S'}$.  Thus $\widetilde{\phi_h}(A_S) = A_{S'}$, so that $\widetilde{\phi_h}$ descends to an isomorphism $H(F,P)/A_S \cong H(F,P)/A_{S'}$.
\end{proof}

\section{Future work and open questions}
\label{sec:questions}

As discussed in Section 2, it is still unknown whether weak universality implies universality. Hence, the following more specialized question is compelling: 
\begin{question}
    Is the isomorphism relation of finitely generated bi-orderable groups universal?
\end{question}

On a different note, our setup to analyze the space of finitely generated left-orderable groups can likely be adapted to analyze other interesting classes of left-orderable groups, so we propose the following general problem:

\begin{problem}
    To analyze the descriptive complexity of the isomorphism relations for other classes of finitely generated left-orderable groups.
\end{problem}

Interestingly, Jay Williams \cite{Wil15} proved that the isomorphism relation of finitely generated \(3\)-step solvable is weakly universal. Thus, the following question is particularly intriguing:

\begin{question}
\label{question : amenable}
    Is the isomorphism relation of finitely generated left-orderable amenable groups (weakly) universal?
\end{question}

Question~\ref{question : amenable} requires new ideas.
In fact, the groups constructed in Theorem~\ref{thm:main} are far from being amenable, as they contain non-abelian free groups. Moreover, Williams' methods seem at odds with orderability.

Last we recall the so-called cocycle property.
Let \(G\) be a Polish group and \(a\colon G\times X \to X\) a Borel action of \(G\) on a Borel set \(X\). We say that this action has the \emph{cocycle property} if there is a Borel map \(\rho\colon E_a \to G\), where \(E_a\), is the orbit equivalence relation defined by \(x\mathbin{E}_ay \iff \exists g(a(g , x) = y)\), such that for \(x\mathbin{E}_ay\), such that 
\begin{enumerate}
\item
\(\rho(x, y)\cdot x=y\)
\item
\(\rho\) is a \emph{cocycle}, i.e., for all \(x,y,z\in X\) in the same orbit we have \(\rho(x, y)=\rho(y, z)\rho(x, y)\).
\end{enumerate}

\begin{question}
\label{question : cocycle property}
    Does \((\mathcal{BO}, \cong)\) have the cocycle property?
\end{question}

It is worth pointing out that the only known equivalence relations which have the cocycle property are all universal. So answering Question~\ref{question : cocycle property} would be interesting either way.

\end{document}